\documentclass{amsart}
\usepackage[utf8]{inputenc}
\usepackage[english]{babel}

\usepackage[margin=2.5cm]{geometry}
\usepackage{enumerate}
\usepackage{float}
\usepackage[final]{graphicx}
\usepackage{epstopdf} 

\usepackage{tabularx}
\newcolumntype{C}[1]{>{\centering\arraybackslash}m{#1}}
\usepackage{ stmaryrd }
\usepackage{caption}
\usepackage{subcaption}
\usepackage{enumitem}
\usepackage{comment}
\usepackage{soul}
\usepackage{pinlabel}

\usepackage{tikz,tikz-cd}
\usepackage{amsmath, multirow}
\usepackage{amsfonts}
\usepackage{amssymb}
\usepackage{amsthm}
\usepackage{hyperref}
\usepackage{cleveref}

\DeclareCaptionSubType[alph]{figure}
\usepackage[T1]{fontenc}
\usepackage{mathrsfs}
\usepackage{mathtools}

\usetikzlibrary{decorations, decorations.markings}
\usetikzlibrary{decorations.pathreplacing}
\usetikzlibrary{arrows,shapes,positioning,patterns,calc,shapes.arrows,intersections}
\usetikzlibrary{knots}
\tikzstyle{none}=[inner sep=0pt]
\pgfdeclarelayer{edgelayer}
\pgfdeclarelayer{nodelayer}
\pgfsetlayers{edgelayer,nodelayer,main}
\usepackage{url} 
\usepackage[all]{xy}

\newtheorem{theorem}{Theorem}[section]
\newtheorem{theorem*}{Theorem}
\newtheorem{proposition}[theorem]{Proposition}
\newtheorem{lemma}[theorem]{Lemma}
\newtheorem{corollary}[theorem]{Corollary}
\newtheorem{corollary*}[theorem]{Corollary}

\newtheorem{remark}[theorem]{Remark}
\newtheorem{example}[theorem]{Example} 
\theoremstyle{definition}

\newtheorem{defn}[theorem]{Definition}

\definecolor{aquamarine}{rgb}{0.5, 1.0, 0.83}
\definecolor{princetonorange}{rgb}{1.0, 0.56, 0.0}
\definecolor{caribbeangreen}{rgb}{0.0, 0.8, 0.6}
\definecolor{bunired}{rgb}{0.8, 0.0, 0.0}
\definecolor{cdgreen}{rgb}{0.0, 0.42, 0.24}
\definecolor{lavender(floral)}{rgb}{0.71, 0.49, 0.86}
\definecolor{bluedefrance}{rgb}{0.19, 0.55, 0.91}
\definecolor{iris}{rgb}{0.35, 0.31, 0.81}
\definecolor{darkgreen}{rgb}{0.33, 0.42, 0.18}

\newcommand{\bN}{\mathbb N}
\newcommand{\bK}{\mathbb K}

\newcommand{\bC}{\mathbb C}
\newcommand{\bZ}{\mathbb Z}

\newcommand{\tC}{{\tt C}}

\newcommand{\tG}{{\tt G}}
\newcommand{\tH}{{\tt H}}

\newcommand{\cP}{\mathcal{P}}

\renewcommand{\H}{\mathrm{H}}

\renewcommand{\bC}{\mathbf{C}}

\newcommand{\Rep}{\mathbf{Rep}}

\newcommand{\Graph}{\mathbf{Graph}}

\newcommand{\cV}{\mathcal{V}}

\newcommand{\cM}{\mathcal{M}}
\renewcommand{\cP}{\mathcal{P}}

\newcommand{\op}{{\rm op}}

\DeclareMathOperator{\Ext}{\mathrm{Ext}}
\DeclareMathOperator{\MH}{\mathrm{MH}}
\DeclareMathOperator{\MC}{\mathrm{MC}}

\DeclareMathOperator{\Hom}{\mathrm{Hom}}

\newcommand{\Mod}{R\text{-{\bf Mod}}}

\address{Luigi Caputi: Dipartimento di Matematica, Universit\`a di Bologna, 40126 Bologna, IT}
\email{luigi.caputi@unibo.it}

\address{Carlo Collari: Dipartimento di Matematica, Universit\`a di Pisa, 56127 Pisa, IT}
\email{carlo.collari@dm.unipi.it}

\title{On finite generation in magnitude (co)homology, and its torsion}
 \author{Luigi Caputi }
 \author{Carlo Collari }

\begin{document}

\begin{abstract}
The aim of this paper is to apply the framework developed by Sam and Snowden to study structural properties of graph homologies, in the spirit of Ramos, Miyata and Proudfoot. 
Our main results concern the magnitude homology of graphs introduced by Hepworth and Willerton, { and we prove that it is a finetely generated functor (on graphs of bounded genus).} More precisely, for graphs of bounded genus, we prove that magnitude cohomology,  in each homological degree, has rank which grows at most polynomially in the number of vertices, and that its torsion is bounded. As a consequence, we obtain  analogous results for path homology of (undirected) graphs. 
\end{abstract}
\maketitle
\section*{Introduction}

Magnitude is a numerical invariant of isometry classes of metric spaces, and it 
finds its motivations in the study of species similarity~\cite{biologymagnitude} -- see~\cite{ Leinsterbook} for a general overview on the topic. 
From the mathematical point of view, the definition of magnitude was motivated, and naturally arose, from considerations on Euler characteristic in enriched category theory \cite{MR3084566} -- see also~\cite{MR3054629, Leinsterbook}.
Aside from category theory and metric geometry, magnitude has also many interesting connections with other areas of mathematics, such as differential geometry~\cite{MR3158044},  minimal surfaces~\cite{MR4520035}, geometric measure theory, and potential theory~\cite{MR3701739}.

Graphs, endowed with the shortest path distance, are prominent examples of metric spaces. 
In this context, magnitude was studied first in~\cite{magnitudelei}. 
Leinster showed that magnitude of graphs has a wealth of interesting properties.
Among those, we have that magnitude is a rational function which can be expressed as an integral power series, that is multiplicative with respect to the Cartesian product of graphs, and that it satisfies  an inclusion-exclusion formula.
Remarkably, 
magnitude shares similar features with -- and yet is not determined by -- the Tutte polynomial, see~\cite{magnitudelei}.

Rather than on magnitude itself, here we focus on its categorification: magnitude (co)ho\-mo\-lo\-gy, as defined by Hepworth and Willerton~\cite{richardHHA,MagnitudeCohomology}, see also~\cite{LeinsterShulman}. 
As with magnitude, also its categorification has  attracted some attention -- see, for instance, the recent papers~\cite{magnitue_discretemorsetheory, asao_girth, SadzanovicSummers, asao}. 
In this context, {categorification} means to associate to a numerical (or polynomial) invariant a whole homology theory, whose Euler characteristic recovers the original invariant. The simplest example of categorification is given by the classical Euler-Poincar\'e characteristic of simplicial complexes, which is categorified by simplicial homology.
The categorification of the Jones polynomial introduced by Khovanov~\cite{Khovanov} has shown the advantage of the  homological and categorical approaches to the study of (polynomial) invariants of graphs, knots, \emph{etc}.
After Khovanov's discovery, the interest in categorification of 
knots and graph invariants, and their properties, skyrocketed.
Among the well-known graph invariants which have been categorified using Khovanov's framework, we can find the chromatic polynomial~\cite{HGRong}, and the Tutte polynomial~\cite{MR2263058}. Magnitude homology follows similar ideas.
A consequence of the categorification process is that it
brings more refined invariants. This is also the case for magnitude homology: there exist graphs with {the} same magnitude, but non-isomorphic magnitude homology groups~\cite[Appendix~A]{magnitue_discretemorsetheory}. Furthermore, the categorical approach allows us to explain some intrinsic properties of the magnitude of graphs, in homological terms; for example, the multiplicative property with respect to Cartesian products descends from a K\"unneth theorem, and a Mayer-Vietoris theorem categorifies the inclusion-exclusion formula.

\subsection*{Statement of results} We now come to the main point of this paper. It was asked in \cite{richardHHA} whether there are graphs with torsion in their magnitude homology. This question was positively answered, first by Kaneta and Yoshinaga~\cite{MR4275098}, and then extended by Sazdanovic and Summers~\cite{SadzanovicSummers}. Computations show that any finitely generated Abelian group appears as a subgroup of the magnitude homology of a graph~\cite[Theorem 3.14]{SadzanovicSummers}. However, to get such torsion, Sazdanovic and Summers increased the combinatorial complexity of the graphs. This means that, to get new torsion, they increase the combinatorial genus of the graphs. It is not clear from their proof whether this behaviour is a structural property of magnitude homology, or {if} it depends on the methods developed in~\cite{MR4275098, SadzanovicSummers}. Instead of working with magnitude homology, here we use its cohomological version~\cite{MagnitudeCohomology}, and we will confine ourselves within the category of finite, connected, undirected graphs. As magnitude homology and cohomology are related by a short exact sequence~\cite[Remark~2.5]{MagnitudeCohomology}, passing to cohomology is not a restriction.

The main goal of this work is to prove that, to get torsion of higher order in magnitude cohomology, it is indeed necessary to increase the combinatorial complexity of the graphs.  Therefore, this behaviour  is somehow due to a structural property of magnitude (co)homology.
To achieve such result, we borrow methods from representation stability~\cite{FI-modules} of (combinatorial) categories,    as recently developed by Sam and Snowden~\cite{Grobner_Cat}. 
The categorical viewpoint enables us to gain a deeper understanding of the behaviour of magnitude (co)homology (and of its torsion), by looking at combinatorial properties of the category of graphs considered. 
Our ideas were inspired by works from Ramos, Miyata and Proudfoot~\cite{  minor_twocolumns, contractioncat,topological_minor_thm}, who proved similar statements for matching complexes and unordered configuration spaces of graphs. 
Recall that the magnitude cohomology of a graph $\tG$, with coefficients in a {commutative} ring $R$, is a bigraded {$R$-}module~$\MH^*_*(\tG;R) = \bigoplus_{k,l} \MH^k_l(\tG;R)$, where~$k$ is  the cohomological grading, and $l$ is the length, {see Definition~\ref{def:magnitude}}. The main result of the paper is the following -- \emph{cf.}~Theorem~\ref{thm:torsionmagnitude};

\begin{theorem*}\label{thm:main}
For every pair of integers $k, g\geq 0$, there exists $m = m(g,k) \in \bZ$ which annihilates the torsion subgroup of $\MH^{k}_{*}(\tG;\bZ)$, for each graph~$\tG$ of genus at most $g$. 
\end{theorem*}

Roughly speaking, the theorem says that the  order of torsion classes in integral magnitude cohomology of connected graphs of genus at most $g$, in a fixed cohomological degree~$k$, is bounded. 
In fact, to apply the ideas of Ramos, Miyata, and Proudfoot to our context, we consider the category $\bC\Graph_{{\leq}g}$ of graphs  with genus at most~$g$ ({see} Remark~\ref{rem:gskghksh}) and contractions. The opposite category of this category is quasi-Gr\"obner in the sense of~\cite{Grobner_Cat}, as proved in~\cite[Theorem~1.1]{contractioncat}.
Let $\cV$ be the ${\bf CGraph}_{\leq g}$-module which associates to a graph $\tG$ the free $R$-module generated by the vertices of $\tG$; we call it the \emph{vertex module} (see Definition~\ref{def:vertexmodule}). Then, in the spirit of~\cite{  minor_twocolumns, contractioncat,topological_minor_thm}, the main technical result of the paper is the following -- cf.~Theorem~\ref{thm:Vkfg}:

\begin{theorem*}\label{thm:tensor_fg} 
For any $k\geq 1$  natural number, the ${\bf CGraph}^{\op}_{\leq g}$-module  ${\rm Hom}_R(\mathcal{V}^{\otimes k},R)$ is finitely generated. 
\end{theorem*} 

Using Theorem~\ref{thm:tensor_fg}, we prove that magnitude cohomology is a finitely generated functor on the category $\bC\Graph^\op_{{\leq}g}$, and the technology developed in~\cite{Grobner_Cat} allows us to infer Theorem~\ref{thm:main}. As a byproduct, we also obtain an estimate  of the growth of the ranks of magnitude cohomology groups -- \emph{cf.}~Corollary~\ref{cor:jafbka};

\begin{theorem*}\label{thm:growth_intro}
    Let $\bK$ be a field, {fix a cohomological grading $k$} and {genus} $g\geq 0$. Then, there exists a polynomial $f\in \bZ[t]$ of degree at most~$g+k+1$, such that, for all {graphs} $\tG$ of genus at most~$ g$, we have
    \[
    \dim_\bK \MH_*^k(\tG;\bK) \leq   f(\# E(\tG)) \ ,
\]
where $\# E(\tG)$ is the number of edges of $\tG$.
\end{theorem*}

This last result says that the growth of the ranks of magnitude cohomology, in a fixed cohomological degree~$k$, is at most polynomial, provided {we} restrict to graphs of bounded genus -- \emph{cf.}~Corollary~\ref{cor:hajfgak}. 
We remark here that the results are consistent with previous computations -- see also Example~\ref{ex:cycle}, and the text thereafter -- for example,
for cycle graphs~\cite{magnitue_discretemorsetheory}. An immediate question is whether the polynomial nature of these estimates also holds for the category of graphs. 
Unfortunately, at the moment, it is not known if the category of graphs $\bC\Graph^\op$ is a quasi-Gr\"{o}bner category. Therefore our result does not extend 
to the whole category of graphs, and an analogue of Theorem~\ref{thm:growth_intro} may or may not hold.
We also remark here
that there are examples of cohomology theories of (directed) graphs whose rank grows exponentially in the number of vertices -- see, for example, the growth rate of multipath cohomology~\cite{primo} with coefficients in an algebra, see~\cite[Table~2]{secondo}.

It was recently shown by Asao~\cite{asao} that magnitude homology is related to another homology theory of  directed graphs, the so-called \emph{path homology}~\cite{Grigoryan_first}. As a corollary of Asao's work, we can directly infer similar results to those above for path (co)homology of {undirected} graphs. For example, we get the following -- see Corollary~\ref{cor:torionPH}.

\begin{theorem*}
For each $g,k$ positive integers, there exists a{n integer} $d = d({g,k}) \in \bZ$ such that, for each graph $\tG$ of genus $g$, the torsion part of the path cohomology $\mathrm{PH}^k(\tG,\bZ)$ 
has exponent at most $d$.
\end{theorem*}

The main ingredient in the proof of the  theorems is that magnitude cohomology is a finitely generated $\bC\Graph_{{\leq}g}^\op$-module, in the sense of representation theory of categories -- \emph{cf.}~Corollary~\ref{cor:mgfg}. 
It would be interesting to see whether, also in the context of directed graphs, the usual homology theories considered in the literature are finitely generated.

\subsection*{Conventions}

All graphs are assumed to be finite, connected and undirected, and are denoted in typewriter font, e.g.~$\tG$. Unless otherwise specified, $R$ will denote a Noetherian commutative ring with identity.
By $\bN$ we will denote the set of non-negative integers.

\subsection*{Acknowledgements}

The authors wish to thank Daniele Celoria, for the helpful discussions. 
The authors are also most grateful to Richard Hepworth and Yasuhiko Asao, for their valuable comments and feedback on the first draft of the paper.  LC  warmly thanks also Anna Pidnebesna and Isa Dallmer-Zerbe.
The authors are grateful to INdAM-GNSAGA. LC~was partially supported by the Starting Grant 101077154 "Definable Algebraic Topology" from the European Research Council. CC~was supported by the MIUR-PRIN project 2017JZ2SW5 and acknowledges the MIUR Excellence Department Project awarded to the Department of Mathematics, University of Pisa, CUP I57G22000700001. 
During the late stages of the writing of this paper, CC was partially supported by INdAM ``mensilità estero'' grant, and wishes to thank Maciej Borodzik and IMPAN for the hospitality.
The authors are thankful to the referees for their comments, which led to an improvement of the paper. We are especially thankful to an anonymous referee who pointed out a mistake in an earlier version of the paper. 

\section{Basic notions}

In this section, we recall some basic notions needed in the follow-up.

\subsection{Finitely generated $\bC$-modules}

We start with some categorical notions as introduced in~\cite{FI-modules} in the context of FI-modules. We  recall the notion of Noetherian modules and  categories, following the general framework developed in~\cite{Grobner_Cat}.

Let $\bC$ be a (essentially) small 
category, and $R\neq 0$ be a commutative unital ring. 
A \emph{representation} of the category~$\bC$, or a \emph{${\bf C}$-module over~$R$}, is a functor $\mathcal{M}\colon\mathbf{C} \to R\text{-{\bf Mod}}$ with values in the category of (left) $R$-modules. A map of $\bC$-modules is a natural transformation of functors. Denote by ${\Rep}_R(\mathbf{C})$ the resulting  category of $\bC$-modules over~$R$. The category~$\Rep_R(\mathbf{C})$ is an Abelian category, and all the classical notions as submodules, kernels, cokernels, injections, surjections, can be defined pointwise. 
For example,~a \emph{submodule} of a $\bC$-module $\mathcal{M}$ is a $\bC$-module $\mathcal{N}\colon \bC\to R\text{-{\bf Mod}}$ such that $\mathcal{N}(c)$ is a submodule of~$\mathcal{M}(c)$, for each object~$c$ of $\bC$, and such that the inclusion maps define a natural transformation. A submodule~$\mathcal{N}$ is called \emph{proper}, if $\mathcal{N}(c)$ is a proper submodule of $\mathcal{M}(c)$ for at least one $c$. Analogously for kernels and cokernels.

For a given $\bC$-module~$\cM$, by an \emph{element} of $\cM$ we mean an element of $\cM(c)$ for some object~$c$ of $\bC$. If $S$ is a subset of the disjoint union $\bigoplus_{c\in \bC}\cM(c)$, the \emph{span}~$\mathrm{span}(S)$ of~$S$ is the minimal $\bC$-submodule of~$\cM$ containing each element of~$S$. We are primarily interested in finitely generated modules, which can be defined as follows:

\begin{defn}\label{def:fgmod}
A $\mathbf{C}$-module $\mathcal{M}$ is \emph{finitely generated} if there exists a finite set of elements 
$
m_1,..., m_k \in \bigoplus_{c\in \bC} \mathcal{M}(c) 
$,
such that $\mathrm{span}(m_1,\dots,m_k)=\cM$. \end{defn}

We can also characterise finitely generated modules in terms of simpler modules.
For each object $c$ of $\bC$,  define a \emph{principal projective} $\bC$-module~$\cP_c$, as follows. The functor  $\cP_c$ is defined on an object~$c'$ of $\bC$ by setting
\[
\cP_c(c')\coloneqq R\, \langle\Hom_{\bC}(c,c')\rangle \ ,
\]
\emph{i.e.}~the free  $R$-module with basis~$\Hom_{\bC}(c,c')$; $\cP_c$ is then extended to morphisms accordingly, by compositions. For a morphism $\gamma\colon c \to c'$, we denote by $e_\gamma$ the corresponding element in $\cP_c(c')$. 
Observe that, for any $\bC$-module~$\cM$, we have a natural isomorphism 
\[  \Hom_{{\Rep}_R(\mathbf{C})}(\cP_c,\cM)\overset{\cong}{\longrightarrow} \cM(c) \ ,\]
which is given by $\phi \mapsto \phi(c)(e_{{\rm id}_c})$. Therefore, for each short exact sequence of $\bC$-modules
\[ 0\to \cM_1 \overset{f}{\to} \cM_2\overset{g}{\to} \cM_3\to 0 \] 
application of the functor $\Hom_{{\Rep}_R(\mathbf{C})}(\cP_c,-)\colon {\Rep}_R(\mathbf{C}) \to R\text{-{\bf Mod}}$ yields a sequence which is naturally equivalent to 
\[ 0\to \cM_1(c)\overset{f(c)}{\to}  \cM_2(c)\overset{g(c)}{\to} \cM_3(c)\to 0\]
and hence exact. It follows that the module~$\cP_c$ is a projective object in~${\Rep}_R(\mathbf{C})$ -- which justifies the name principal projective.  Finitely generated $\bC$-modules can be characterised in terms of principal projectives,  \emph{cf.}~\cite[Proposition~2.3]{FI_Noetherian}:

\begin{lemma}\label{lem:fingengen}
A $\bC$-module $\cM$ is finitely generated if and only if there exists a surjection
    \[
\bigoplus_{i=1}^n \cP_{c_i} \to \cM
    \]
for some objects $c_1,\dots,c_n$ of $\bC$.
\end{lemma}

For a finitely generated $\bC$-module $\cM$, we will also refer to the objects  $c_1,\dots,c_n$ of $\bC$ in the lemma as \emph{generators} of $\cM$.
We now recall the notion of Noetherianity in the context of modules over an arbitrary category, which is central for this work.

\begin{defn}
    A $\mathbf{C}$-module  is \emph{Noetherian} if all its submodules are finitely generated. The category~${\Rep}_R(\mathbf{C})$ is Noetherian if all its finitely generated $\mathbf{C}$-modules  are Noetherian. 
\end{defn}

Observe that, in discussing properties related to finite generation, it is often possible to restrict to principal projective modules. Indeed, by \cite[Proposition~3.1.1]{Grobner_Cat}, the category~${\Rep}_R(\mathbf{C})$ is Noetherian if and only if every principal projective module $\cP_c$ is Noetherian.

\begin{example}\label{ex:FI}
Denote by $\mathbf{FI}$  the category of finite sets and injective functions. Then, the category~${\Rep}_R(\mathbf{FI})$ is Noetherian by~\cite[Theorem A]{FI_Noetherian} for any Noetherian ring~$R$.
\end{example}

The category $\mathbf{FI}$ is an example of an $EI$-category, \emph{i.e.}~a category in which every endomorphism is an isomorphism. It was proven  that~${\Rep}_R(\mathbf{C})$ is Noetherian, for any finite $EI$-category $\mathbf{C}$ and Noetherian ring~$R$~\cite[Lemma~16.10]{MR1027600}. This result was further extended to infinite $EI$-categories (satisfying some mild combinatorial conditions) in~\cite[Theorem~3.7]{MR3358549}.

\begin{example}
Let $\mathbf{FA}$ be the category  of finite sets, and all functions. Then, the category~${\Rep}_R(\mathbf{FA})$ is Noetherian~\cite[Corollary~7.3.5]{Grobner_Cat} for any  Noetherian commutative unital ring $R$.
\end{example}

Following the seminal paper \cite{FI-modules}, Noetherian properties of various other categories have been extensively investigated, in particular thanks to the techniques developed by Sam and Snowden in \cite{Grobner_Cat}. One of the main results in the latter paper is that, for a given category~$\bC$ (with some combinatorial assumptions), and a (possibly non-commutative) left Noetherian ring $R$, the associated category of representations is Noetherian as well.

\subsection{Graph categories}

The aim of this section is to discuss the combinatorial properties of certain categories of graphs, from the viewpoint of~\cite{contractioncat}.

In the following, by a graph~$\tG$ we mean a finite, connected and non-empty $1$-dimensional CW-complex; it has a set of vertices~$V(\tG)$ and a set of edges~$E(\tG)$, which are unordered pairs of vertices, possibly with multiplicities. 
We recall the intuitive notions of contractions, deletions and minor morphisms; for a more detailed account of these operations, we refer to~\cite{minor_twocolumns,topological_minor_thm}.

Let $\tG$ be a graph and $e\in E(\tG)$ be an edge. The \emph{contraction} of~$\tG$ with respect to the edge $e$, is the graph $\tG/e$ obtained from $\tG$ by contracting $e$ to a point.
The \emph{deletion} of $e$ is the graph $\tG\setminus e$ obtained from $\tG$ by removing $e$ from the set of edges of $\tG$. 
Note that the operation of contracting edges does not change the homotopy type of $\tG$, unless the edge contracted is a self-loop.
We are not going to consider this latter case, and only allow contractions of edges with distinct endpoints. Similarly, when dealing with deletions, we will only allow deletions of graphs for which $\tG\setminus e$ is connected.
A \emph{minor} of a graph $\tG'$ is a graph~$\tG$ that is isomorphic to a graph obtained from~$\tG'$  by iterative contractions and deletions. 
More formally, we have the following definition of minor morphism of graphs.

\begin{defn}[{\cite[Definition~2.1]{minor_twocolumns}}]
A \emph{minor morphism} $\phi\colon \tG'\to\tG$ is a map of sets
\[
\phi\colon V(\tG')\sqcup E(\tG')\sqcup \{\star\}\to V(\tG)\sqcup E(\tG)\sqcup \{\star\} \ ,
\]
such that:
\begin{itemize}
\item $\phi(V(\tG'))=V(\tG)$ and $\phi(\star)=\star$;
\item if an edge $e \in E(\tG')$ has endpoints $\{v, w\}$,  
and $\phi(e)\neq \star$,
then either $\phi(e) = \phi(v) = \phi(w)$ is a vertex of $\tG$,
or $\phi(e)$ is an edge of $\tG$ with endpoints $\phi(v) $ and $ \phi(w)$;
\item $\phi$ maps  $\phi^{-1}(E(\tG))$ bijectively onto $E(\tG)$;
\item for each vertex $v\in \tG$, the preimage $\phi^{-1}(v)$ (as a subgraph of $\tG'$) is a tree.
    \end{itemize}
\end{defn}

The preimage of $\star$ under $\phi$ consists of deleted edges, whereas the edges that are mapped to vertices of $\tG$ represent the contracted ones. 
Furthermore, the last item in the definition implies that self-loops cannot be contracted, but only deleted. 

In the follow-up, we will mainly consider (subcategories of) the category~$\Graph$ 
of finite non-empty   graphs, with minor morphisms of  graphs. For example,
we consider the subcategory $\bC\Graph$ of $\Graph$ consisting of finite, non-empty,  connected,   graphs, where the morphisms are given by contractions.

For a graph $\tG$, we call the (combinatorial) \emph{genus} of~$\tG$ the quantity $ |E(\tG)|-|V(\tG)|+1$. This is often called \emph{circuit rank} but, for consistency, here we prefer to follow the terminology adopted in~\cite{contractioncat}.

\begin{remark}\label{rem:gskghksh}
Consider the full subcategory $\bC\Graph_g$ of $\bC\Graph$ given by graphs of genus $g$. In particular,     as contractions do not change the genus of a graph,   $\bC\Graph$ can be seen as the disjoint union of the categories $\bC\Graph_g$, for $g\in \bN$. We will also denote by $\bC\Graph_{\leq g}$ the subcategory of $\bC\Graph$ spanned by graphs of genus $\leq g$. The category $\bC\Graph_g$ was denoted by $\mathcal{G}_g$ in \cite{contractioncat}. 
\end{remark}

We have the following key result:

\begin{theorem}[{\cite[Theorem~1.1]{contractioncat}}]
\label{cor:contractcatfg}
For any $g\geq 0$, the representation category of ${\bf CGraph}^{\op}_{\leq g}$ is Noetherian over any Noetherian ring. 
\end{theorem}

Thus, for any commutative unital Noetherian ring $R$, all submodules of finitely generated ${\bf CGraph}^{\op}_{\leq g}$ modules 
are finitely generated.

\section{Magnitude (co)homology }

In this section, we prove the main result of the paper, that is, that magnitude cohomology of graphs is a finitely generated functor on the  category $\bC\Graph^\op_{\leq g}$ of connected graphs with bounded genus, and contractions. As a consequence, in the next section, we shall recover structural results on its torsion and rank growth. 
 
\subsection{Magnitude homology and  cohomology of graphs}
We start with recalling the definition of magnitude homology of graphs, and we will then focus on magnitude cohomology. We will mainly follow~\cite{richardHHA, asao,  MagnitudeCohomology}.

First, observe that a connected graph can be seen as a metric space with the path metric -- the distance between two vertices of the graph being given by the length of a shortest path in the graph connecting them. 
To be more precise, the points of the metric space associated to a graph $\tG$ are its vertices, and edges are declared to have length~$1$. Concretely, the metric~$d_{\tG}$ on $\tG$ is  given by
\begin{align*}
    d_{\tG}(v,w)\coloneqq \min  \{ d_{\tG}(v,v_1) + &\sum_{i=1}^{k-2} d_{\tG}(v_i,v_{i+1}) + d_{\tG}(v_{k-1},w) \mid \\
    & \{v,v_1\}, \{v_{k-1},w\}, \{v_i,v_{i+1}\}\in E(\tG), i=1,\dots, k-2 \}
\end{align*}
for $v,w$ vertices of $\tG$. If the graph is not connected, we set $d_{\tG}(v,w)\coloneqq \infty$ for every $v,w\in \tG$ not connected by any path; hence, for non-connected graphs, $d_{\tG}$ is an extended metric. For a $(k+1)$-tuple$(v_0,\dots,v_k)$ of vertices of $\tG$, with $v_i\neq v_{i+1}$ and $d_{\tG}(v_i,v_{i+1})<\infty$ for each~$i$, the \emph{length} of $(v_0,\dots,v_k)$ in $\tG$ is the number 
\[
\ell(v_0,\dots,v_k)\coloneqq \sum_{i=0}^{k-1}d_{\tG}(v_i,v_{i+1}) \ .
\]
We  recall the definition of magnitude chain groups. Let $R$ be a commutative unital ring.

\begin{defn}\label{def:magnitude}
    Let $\tG$ be a graph, and $l,k\in \bN$ natural numbers. We let
    \[
    \MC_{k,l}(\tG;R)\coloneqq R\langle (v_0,\dots,v_k) \mid v_0\neq \dots\neq v_k, \ell(v_0,\dots,v_k)=l \rangle 
    \]
    be the free $R$-module on the paths of length $l$ on $k+1$ vertices  of $\tG$. 
    The differential
    $
    \partial\colon\MC_{k,l}(\tG;R)\to\MC_{k-1,l}(\tG;R)
    $
    is defined on $(k+1)$-uples $(v_0,\dots,v_k)$ by
    \[
    \partial(v_0,\dots,v_k)\coloneqq \sum_{i=1}^{k-1} (-1)^i\partial_i(v_0,\dots,v_k) \ ,
    \]
    where $\partial_i(v_0,\dots,v_k)=(v_0,\dots,v_{i-1},v_{i+1}, \dots, v_k)$ assuming that $\ell(v_0,\dots,v_k)=l=\ell(v_0,\dots,v_{i-1},v_{i+1}, \dots, v_k)$, and $\partial_i$ is set to be $0$ otherwise. 
\end{defn}

For a given $l\in \bN$, the pair $(\MC_{*,l}(\tG;R), \partial)$ is a chain complex by \cite[Lemma~2.11]{richardHHA}.

\begin{defn}[{\cite[Definition~2.4]{richardHHA}}]
The \emph{magnitude homology}~$\MH_{*,*}(\tG;R)$ of $\tG$ is defined as the bigraded $R$-module $\bigoplus_{k,l\geq 0} \MH_{k,l}(\tG;R)$, where
    \[
    \MH_{k,l}(\tG;R)\coloneqq \H_k(\MC_{*,l}(\tG;R), \partial)
    \]
    is given by the homology of the magnitude chain complex.
\end{defn}

We can reinterpret the definition of magnitude homology groups as follows -- see also \cite[Section~2]{asao} for the case of directed graphs.
Given a non negative integer $k$, define $\Lambda_k(\tG;R)$ as the $R$-module freely generated by the $(k+1)$-uples of vertices of $\tG$, that is
$$\Lambda_k(\tG;R)\coloneqq R\langle (v_0,\dots, v_k)\mid v_i\in V(\tG)\rangle\ .$$
Consider the submodule 
$
I_k(\tG;R)\coloneqq R\langle (v_0,\dots, v_k)\mid v_i=v_{i+1} \text{ for some } i \rangle
$
of $\Lambda_k(\tG;R)$, where $I_0(\tG)$ is set to $ 0$.   Note that these modules can be equipped with a differential~${\rm d}$, defined analogously to $\partial$, 
which makes them
chain complexes. Since~$I_k(\tG;R)\subseteq \Lambda_k(\tG;R)$, we can form the quotient chain complex with modules $R_k(\tG;R)\coloneqq \Lambda_k(\tG;R)/ I_k(\tG;R)$. Then, confining ourselves to the setting of undirected graphs, we get that the $R$-module 
$\MC_{k,l}(\tG;R)$ can also be defined as the submodule of $R_k(\tG;R)$ given by $(k+1)$-uples of vertices whose length is precisely~$l$;
this is compatible with the chain complex structure -- \emph{cf.}~\cite[Lemma~2.14]{asao}. In particular, we get the following:

\begin{remark}
    The magnitude chain modules $\MC_{k,l}(\tG;R)$ are sub-quotients of the free $R$-module $\Lambda_k(\tG;R)$.
\end{remark}

It is possible to modify the differential $\delta$ so to get a new differential $$\delta'\colon \MC_{k,l}(\tG;R)\to \MC_{k-1,l-1}(\tG;R) \ , $$ in turn inducing an homomorphism $\partial'\colon  \MH_{k,l}(\tG;R)\to \MH_{k-1,l-1}(\tG;R)$ between the magnitude homology groups. Equipped with this new differential, also  magnitude homology can be seen as a chain complex $(\MH_{k-*,l-*}(\tG;R),\partial')$. The homology of the resulting chain complex was denoted by $\mathcal{MH}_{k-*}^{l-*}(\tG)$ in \cite[Definition~2.21]{asao}. 

\begin{remark}[{\cite[Proposition~6.11]{asao}}]\label{rem:pathhom}
For $k=l$, the homology theory $\mathcal{MH}_{k-*}^{l-*}(\tG)$, for directed graphs, recovers the (reduced) path homology of directed graphs introduced in~\cite{Grigoryan_first}. The same proof, in the undirected setting, produces an isomorphism of $\mathcal{MH}_{k-*}^{l-*}(\tG)$ for undirected graphs, with the reduced path homology of graphs -- see, e.g.~\cite[Section~2.2]{barcelo} for the definition. 
\end{remark}

Magnitude homology is an homology theory of (directed) graphs, and it is functorial with respect to contractions. Recall that a contraction of a graph $\tG$ with respect to an edge~$e$ is the graph obtained from $\tG$ by contracting $e$ to a point. More specifically, if $\tG$ and $\tH$ are graphs, consider maps $\phi\colon \tG\to\tH$ of vertices that preserve or contract each edge of~$\tG$. Observe that such maps do not increase the length of tuples of vertices of $\tG$: that is we have
$
\ell(\phi(v_0),\dots,\phi(v_k))\leq \ell(v_0,\dots,v_k) 
$.
Every contraction $\phi\colon \tG\to\tH$ of graphs induces a chain map
\[
\phi_\# \colon \MC_{*,*}(\tG;R)\to \MC_{*,*}(\tH;R) 
\]
which to a tuple $(v_0,\dots,v_k)$ of $\tG$ associates the tuple $(\phi(v_0),\dots,\phi(v_k))$ if {the length} $\ell(\phi(v_0),\dots,\phi(v_k))$ {equals the length} $\ell(v_0,\dots,v_k)$, and it is set to be $0$ otherwise. The map $\phi_\#$ is a chain map, as it commutes with the differential $\delta$, and it induces a map
in magnitude homology. Recall that we denote by $\bC\Graph$ the category of graphs and contractions.

\begin{proposition}[{\cite[Proposition~3.3]{richardHHA}}]
    Magnitude homology is a functor
    \[
    \MH_{*,*}\colon \bC\Graph\to \mathbf{BiGrMod}_R
    \]
    from the contraction category of graphs to the category of bigraded $R$-modules.
\end{proposition}

By dualising the definition of magnitude homology, as customary, we get the definition of magnitude cohomology -- see \cite{MagnitudeCohomology} -- which we now recall:

\begin{defn}
\emph{Magnitude cohomology} ${\rm MH}^{*}_{*}$ is the cohomology of the  complex
\[ {\rm MC}_{l}^{k}(\tG; R) = {\rm Hom}({\rm MC}_{l,k}(\tG; R); R)\ ,\]    
equipped with the dual differential.
\end{defn}

This defines a functor with respect to the dual maps inducing functoriality in magnitude cohomology --  \emph{cf.}~\cite[Definition~2.2]{MagnitudeCohomology}. 
In particular, the dualisation defines a functor 
    \[
    \MH_{*}^{*}\colon {\bf CGraph}^{\op} \to \mathbf{BiGrMod}_R \ .
    \]
We recall that in ${\bf CGraph}^{\op}$ there is a morphism $\tG\to\tG'$ if, and only if, the graph $\tG$ is obtained from the graph~$\tG'$ by a sequence of contractions. We conclude the section observing that magnitude homology and magnitude cohomology are related by a universal coefficients short exact sequence by \cite[Remark~2.5]{MagnitudeCohomology}. Thanks to this short exact sequence, results on magnitude homology of graphs can be derived from results on magnitude cohomology, as customary. In the next subsections, we will restrict ourselves to the case of  magnitude cohomology. 

\subsection{Magnitude cohomology is finitely generated}
In order to prove that magnitude cohomology (in a fixed $k$-degree) is finitely generated, we exhibit a finitely generated module such that magnitude cohomology is a subquotient of it. 

Recall that  $\Graph$ denotes the category of graphs and minor morphisms and let $R$ be a commutative unital Noetherian ring.

\begin{defn}\label{def:vertexmodule}
 The \emph{vertex module} is the  functor 
$ \mathcal{V}\colon {\bf Graph} \to \Mod$,
which assigns to each graph $\tG$ the $R$-module
\[\mathcal{V}({\tG}) = R\left\langle v \mid v\in V(\tG) \right\rangle\]
freely generated by the vertices of $\tG$. 
To each minor morphism $\phi \colon \tG \to \tG'$, that is if $\tG'$ is obtained from $\tG$ via contractions and deletions, it assigns the map~$ \mathcal{V}(\phi)\colon \mathcal{V}({\tG}) {\longrightarrow} \mathcal{V}({\tG}')$ given by $ {v} \mapsto {\phi(v)} $.  
\end{defn}

Consider the restriction of the module $\mathcal{V}^{\otimes k}$ to the category $\bC\Graph$ of graphs and contractions. Denote by $ \MH_{k,*}(-;R)$ the sum $\bigoplus_{l}\MH_{k,l}(-;R)$.

\begin{lemma}\label{lemma:magnitudesubq}
For each $k$, the magnitude homology $ \MH_{k,*}(-;R)$ is a sub-quotient of the module~$\mathcal{V}^{\otimes k+1}$.
\end{lemma}
\begin{proof}
It is sufficient to notice that $\Lambda_k(\tG)$ can be identified with $ \mathcal{V}^{\otimes k+1}(\tG)$, by identifying~$(v_0,...,v_k)$ with the elementary tensor $v_0\otimes ... \otimes v_k$.
The maps which induce the functoriality in magnitude homology are induced by the identification of the vertices of the contracted edges. 
Taking the directed sum over all lengths~$l$, we get the result.
\end{proof}

Magnitude cohomology ${\rm MH}^{*}_{*}$ is by definition  the cohomology of the dual complex of the  magnitude homology of graphs. If a module $M$ is a subquotient of a module $N$, then it is not generally true that ${\rm Hom}(M,R)$ is a subquotient of ${\rm Hom}(N,R)$. However, this is the case for magnitude cohomology; in fact, we have the following:   

\begin{proposition}
    \label{lem:MHsubq}
For each non-negative integer $k$, magnitude cohomology $ \MH^{k}_{*}(-;R)=\bigoplus_l \MH^{k}_{l}(-;R)$ is a sub-quotient of ${\rm Hom}_R(\mathcal{V}^{\otimes k + 1}, R)$.
\end{proposition}

\begin{proof}
    Fix $l$ a natural number and consider the magnitude chain group $\MC_{k,l}(\tG;R)$ of a graph $\tG$. This is a sub-quotient of $\mathcal{V}^{\otimes k+ 1}(\tG)$ by Lemma~\ref{lemma:magnitudesubq}. Observe that the modules $\Lambda_k(\tG;R), I_k(\tG;R)$ and $R_k(\tG;R)$ are free $R$-modules.  Thus,  
   ${\rm Hom}_R(R_k(\tG;R),R)$ can be seen as a submodule of ${\rm Hom}_R(\mathcal{V}^{\otimes k+ 1}(\tG;R),R) $ -- since $\Lambda_k(\tG;R) \cong \mathcal{V}^{\otimes k+ 1}(\tG)$ canonically. By definition, $R_k(\tG;R)\coloneqq \Lambda_k(\tG;R)/ I_k(\tG;R)$ consists of tuples of subsequently distinct elements of $V(\tG)$, and $\MC_{k,l}(\tG;R)$ is given by the tuples $(v_0,\dots, v_k)$ in $R_k(\tG;R)$ of length~$l$. Using the parameter $l$, we can filter the free module $R_k(\tG;R)$. Using such filtration, also the magnitude homology groups $\MC_{k,l}(\tG;R)$ are spanned by  elements of a basis of $R_k(\tG;R)$. 
    
    We have a short exact sequence
    \begin{equation}\label{eq:ses MC-R}
    0\to \MC_{k,l}(\tG;R)\to R_k(\tG;R) \to Q_k(\tG;R) \to 0
    \end{equation}
    where $Q_k(\tG;R)$ is the associated quotient, which is also a free $R$-module. 
    In particular, the group $\Ext^i_R(Q_k(\tG;R), M)$ are zero for all $i\geq 1$ and $R$-module $M$. Hence, dualising the short exact sequence in Eq.~\eqref{eq:ses MC-R}, we get the short exact sequence
    \begin{equation*}
    0\to {\Hom}_R( Q_k(\tG;R),R) \to {\rm Hom}_R(R_k(\tG;R),R) \to \underbrace{{\rm Hom}_R(\MC_{k,l}(\tG;R),R)}_{ = \MC^{k}_{l}(\tG;R))}  \to 0
    \end{equation*}
    from which the statement follows.
\end{proof}

Our goal is  now to show that the module  ${\rm Hom}_R(\mathcal{V}^{\otimes k},R)$ is finitely generated, for all~$k>0$. As we do not know whether the full category of graphs and minor morphisms, or simply contractions, is Noetherian, we shall restrict to the subcategory of graphs with bounded genus.

For any fixed graph $\tG$ the $R$-module ${\rm Hom}_R(\mathcal{V}^{\otimes k}(\tG),R)$ is freely generated (as  $R$-module) by the functions
$ \delta^{\tG}_{v_1 \otimes \cdots \otimes v_{k}}\colon\mathcal{V}^{\otimes k}(\tG) \longrightarrow R$,
defined by
\[\delta^{\tG}_{v_1 \otimes \cdots \otimes v_{k}}(w_1 \otimes \cdots \otimes w_{k}) = \begin{cases}
1 & \text{if } w_i = v_i\text{ for all }i,\\
0 & \text{otherwise.}
\end{cases}\]
Denote by 
\[ M_k(\tG) = {\rm span}_R(\psi_c(\delta^{\tG'}_{w_1\otimes \dots \otimes w_k}) \mid c\in {\rm Hom}_{{\bf CGraph}^{\op}}(\tG',\tG), \tG'\neq \tG,  w_i\in V(\tG')\ \forall i)\ ,\]
where $\psi_c \colon {\rm Hom}_R(\mathcal{V}^{\otimes k}(\tG'),R) \to {\rm Hom}_R(\mathcal{V}^{\otimes k}(\tG),R)$ denotes the map induced by the contraction $c\colon \tG\to\tG'$. More precisely, $\psi_c$ is just the pre-composition with the map $\cV(c)\colon \cV(\tG) \to \cV(\tG')$. 
We start by proving the following:

\begin{proposition}\label{prop:finite gen}
Given a connected graph $\tG$ with more than $k+1$ vertices, we have that 
$ {\rm Hom}_R(\mathcal{V}^{\otimes k}(\tG),R) =  M_k(\tG)$.
\end{proposition}

\begin{proof}
Pick $v_1,\dots, v_k\in V(\tG)$ 
and  indices $i_1,\dots,i_n$ such that the vertices $v_{i_1},\dots,v_{i_n}$ are distinct and  $\{ v_{i_1},\dots,v_{i_n}\} = \{v_1,\dots, v_k\}$. 
First, we need to establish some notation. Set $I_j \coloneqq \{ t \mid v_{t} = v_{i_j}\}$, for $j=1,\dots,n$, and 
\[ a_i = a_i(v_1,\dots, v_k) = \vert \{ j \mid \vert I_j\vert  = i \}  \vert\ , \]
that is $a_i$ is the number of $v_{i_j}$s which appear exactly $i$ times among $v_1,\dots, v_k$.
Set also 
\[ l = l(v_1,\dots, v_k) = \min\{d_{\tG}(v,w) \mid v, w \in V(\tG) \setminus \{v_1,\dots, v_k\}\}\ ,\]
with $d_{\tG}$ is the shortest path metric.
With this notation in place, we can associate to each $k$-tuple $(v_1,\dots, v_k)$ an element $\chi(v_1,\dots, v_k) \in \bN^{k+1}$ by setting
$ \chi(v_1,\dots, v_k)  \coloneqq (a_k, \dots, a_1, l)$.
The set of all possible values of $\chi$ is totally ordered by the lexicographic order $<_{\rm lex}$. 
We shall prove by recursion on the values of $\chi$ that  
\begin{equation}
\label{eq:statement rec}
\delta^{\tG}_{v_1 \otimes \cdots \otimes v_{k}} \in M_k(\tG)\ , 
\end{equation}
for all choices of $v_1,\dots, v_k\in V(\tG)$.  As $M_k(\tG)$ is by definition a submodule of the~$R$-module $ {\rm Hom}_R(\mathcal{V}^{\otimes k}(\tG),R)$, the statement will follow.

First, we prove the claim for all $(v_1,\dots, v_k)$ such that
$ \chi(v_1,\dots, v_k)  = (0,\dots,0,k,1)$,
which is the minimum possible value of $\chi$.
In this case, there are two vertices $v, w \in V(\tG) \setminus \{v_1,\dots, v_k\}$
which are the endpoints of an edge $e$. Then, if $c$ denotes the contraction of $e$, we have
\[  \psi_{c}(\delta^{\tG'}_{\overline{v}_1\otimes \dots \otimes \overline{v}_{k}}) = \delta^{\tG}_{v_1\otimes \dots \otimes v_{k}}\ , \]
where $\overline{v}_{k} \coloneqq c(v_k)$. 
This proves \eqref{eq:statement rec} in this case.

We recursively assume that \eqref{eq:statement rec} holds for all $(w_1,\dots,w_k)$ such that 
$\chi(w_1,\dots,w_k) <_{\rm lex} \chi(v_1,\dots,v_k)$.
Let $v, w \in V(\tG) \setminus \{v_1,\dots, v_k\}$ be  such that $d_{\tG}(v,w) = l(v_1,\dots,v_k)$. We have either~$d_{\tG}(v,w) = 1$ or $d_{\tG}(v,w) = l>1$. In the former case, with the same reasoning used for the minimal values of $\chi$, we get that \eqref{eq:statement rec} holds. Thus, we can assume $l>1$. In this case, any minimal path between $v$ and $w$  contains only $v_i$s. Let $v_{i_s}$ be the first vertex after $v$ on such a minimal path. 
Then, if $c$ denotes the contraction of an edge between $v$ and $v_{i_s}$, by definition of $\psi_c$ we have
\begin{equation}\label{eq:jfhja}
    \psi_{c}(\delta^{\tG'}_{\overline{v}_1\otimes \dots \otimes \overline{v}_{k}}) = \sum_{S\subseteq I_s} \delta^{\tG}_{x^{S}_1 \otimes \dots \otimes x^{S}_k}\ ,
\end{equation} 
where $x^{S}_j$ equals $v$ if $j\in S$, and $v_{j}$ otherwise.
Clearly, $x^{\emptyset}_1 \otimes \dots \otimes x^{\emptyset}_k = v_1 \otimes \dots\otimes v_k$.
Therefore we can rewrite~\eqref{eq:jfhja} as:
\begin{equation}\label{eq:kgafal}
    \delta^{\tG}_{v_1\otimes \dots \otimes v_{k}} =  \psi_{c}(\delta^{\tG'}_{\overline{v}_1\otimes \dots \otimes \overline{v}_{k}}) - \sum_{\tiny \begin{matrix}S\subseteq I_s\\ S\neq \emptyset \end{matrix}} \delta^{\tG}_{x^{S}_1 \otimes \dots \otimes x^{S}_k}\ .
\end{equation} 
We now claim that for all $S\subseteq I_s$ non-empty
$\chi(x^{S}_1, \dots , x^{S}_k) <_{\rm lex}   \chi(v_1,\dots,v_k)$.
If $S\neq \emptyset, I_s$, then $a_i(x^{S}_1, \dots , x^{S}_k) = a_i(v_1,\dots,v_k)$ for $i> \vert I_s \vert$, since all $v_j$s but $v_{i_s}$ appear with the same multiplicity -- and~$v$ and $v_{i_s}$ appear with multiplicity at most $\vert I_s \vert$. On the other hand, $a_{\vert I_s \vert}(x^{S}_1, \dots , x^{S}_k) < a_{\vert I_s \vert}(v_1,\dots,v_k)$ because $v_{i_s}$ appears now $\vert I_s \setminus S \vert $ times and $v$ appears $\vert S \vert$ times.
If $S = I_s$, then the $a_i$s are unchanged (since $v_{i_s}$ does not appear anymore, and $v$ appears $\vert I_s \vert$ times). Nonetheless, we have 
\[ l(x^{I_s}_1, \dots , x^{I_s}_k) \leq d_{\tG}(v_{i_s},w) = d_{\tG}(v,w) - 1 = l(v_1,\dots,v_k) -1 < l(v_1,\dots,v_k) \ ,\]
and thus $\chi(x^{S}_1, \dots , x^{S}_k) <_{\rm lex}   \chi(v_1,\dots,v_k)$, as claimed. 
Therefore, by recursive hypothesis we get
\[\sum_{\tiny \begin{matrix}S\subseteq I_s\\ S\neq \emptyset \end{matrix}} \delta^{\tG}_{x^{S}_1 \otimes \dots \otimes x^{S}_k} \in M_k(\tG)\ . \]
Thus, from \eqref{eq:kgafal}, we obtain that \eqref{eq:statement rec} also holds for $l>1$, concluding the proof. 
\end{proof}

We are now ready to prove the main theorem of this section. Fix an integer $k\geq 1$.

\begin{theorem}\label{thm:Vkfg}
The ${\bf CGraph}^{\op}_{\leq g}$-module  ${\rm Hom}_R(\mathcal{V}^{\otimes k},R)$ is finitely generated. Moreover, a set of generators is given by all graphs of genus at most $g$ with at most $k+1$ vertices.
\end{theorem}
\begin{proof}
Let $N_k$ be the sub-module  of the ${\bf CGraph}^{\op}_{\leq g}$-module ${\rm Hom}_R(\cV^{\otimes k}(-);R)$ spanned by the set~$ S = \{ \delta^{\tG}_{v_1\otimes \cdots \otimes v_k}\}$, where $\tG$ ranges among all graphs of genus at most $g$ with at most $k+1$ vertices, and~$v_1, \dots , v_k$ range among the vertices of $\tG$. These graphs are a finite number since the genus is bounded, and hence $S$ is a finite set.   

Clearly, ~$N_k(\tG) = {\rm Hom}_R(\cV^{\otimes k}(\tG);R)$ for all graphs with at most $k+1$ vertices.

We proceed by induction on the number of vertices to cover the remaining cases.
Suppose that, whenever $\tG'$ has less than $r >  k+1$ vertices~$N_k(\tG') = {\rm Hom}_R(\cV^{\otimes k}(\tG');R)$. Then, by Proposition~\ref{prop:finite gen}, every $\phi \in {\rm Hom}(\cV^{\otimes k}(\tG);R)$ can be written as 
\[ \phi = \sum_{i=1}^{n} \alpha_i \psi_{c_i}(\delta^{\tG_i}_{v^i_1\otimes \cdots \otimes v^i_k})\ ,\]
for some contractions $c_i\colon \tG \to \tG_i$, and some  $v^i_1, \dots , v^i_k \in V(\tG_i)$. Since the $c_i$s are contractions, $\tG_i$ has at most $r-1$ vertices for each $i$. Hence, $\delta^{\tG_i}_{v^i_1\otimes \cdots \otimes v^i_k} \in N_{k}(\tG_i)$ by inductive hypothesis. 
Since $\psi_{c_i}(N_{k}(\tG_i)) \subseteq N_{k}(\tG)$ for each $i$, the first part of the statement follows. To conclude, note that the generators of the module $N_k$ are all graphs of genus at most $g$ with at most $k+1$ vertices, in view of  Lemma~\ref{lem:fingengen}.
\end{proof}

\begin{corollary}\label{cor:mgfg}
The ${\bf CGraph}^{\op}_{\leq g}$-module  ${\rm MH}^{k}_{l}(-;R)$ is finitely generated.    
\end{corollary} 

\begin{proof}
  In view of Theorem~\ref{cor:contractcatfg}, sub-quotients of finitely generated ${\bf CGraph}^{\op}_{\leq g}$-modules are finitely generated. 
  By Theorem~\ref{thm:Vkfg}, the ${\bf CGraph}^{\op}_{\leq g}$-module ${\rm Hom}_R(\mathcal{V}^{\otimes k+1},R)$  is finitely generated. Therefore, the statement  follows from Proposition~\ref{lem:MHsubq}.
 \end{proof}

\section{Applications}

 A first consequence of the finite generation property described in Corollary~\ref{cor:mgfg} is a bound on the magnitude (co)homology ranks that depends on the number of edges of the considered graphs.  
Our first corollary is a general result for graphs of bounded genus, and the second is related to subsequent subdivisions of the  edges.

\begin{corollary}\label{cor:jafbka}
    Let $\bK$ be a field, and $g\geq 0$. Then, there exists a polynomial $f\in \bZ[t]$ of degree at most~$g+k+1$, such that, for all $\tG$ of genus at most $ g$, we have
    \[
    \dim_\bK \MH_*^k(\tG;\bK) \leq   f(\# E(\tG)) \ ,
\]
where $\# E(\tG)$ is the number of edges of $\tG$.
\end{corollary}

\begin{proof}
    The statement directly follows from \cite[Proposition 4.3]{contractioncat} after noticing that magnitude cohomology, in cohomological degree~$k$ and independently on the length degree~$l$, is a subquotient of a module which is finitely generated in degree $\leq g + k +1$ -- since it is generated  by graphs of genus $\leq g$ and at most $k+2$ vertices, cf.~Theorem~\ref{thm:Vkfg}. 
\end{proof}

Let $\tG$ be a graph of genus~$g$, $\underline{e}=(e_1,\dots,e_r)$ a tuple of distinct edges of $\tG$ which are not self-loops. We fix a direction on $e_1,\dots,e_r$.  This extra data is auxiliary, that is the choice of the direction is immaterial, but it is needed to explicitly write down the functor below. For a tuple  $\underline{m}=(m_i,\dots, m_r)$ of non-negative integers, we let $\tG(\underline{e},\underline{m})$ be the graph obtained from~$\tG$ by subdividing each edge $e_i$ a number of $m_i$ times. If $m_i=0$, then ``subdivision'' means ``contraction'' of the edge $e_i$. Denote by  $\mathbf{OI}$   the category of linearly ordered finite sets and ordered inclusions. Consider the product category $\mathbf{OI}^r$.  The directions on the edges $e_i$ have been chosen in order to construct a subdivision functor
\[
\Phi_{\tG,\underline{e}}\colon \mathbf{OI}^r\to {\bf CGraph}^{\op}_{\leq g}
\]
which associates to a linearly ordered set $[\underline{m}]\in\mathbf{OI}^r$ the graph $\tG(\underline{e},\underline{m})$, and to a morphism $[\underline{m}]\to [\underline{n}]$ in $\mathbf{OI}^r$ a contraction $\tG(\underline{e},\underline{n})\to \tG(\underline{e},\underline{m})$ -- see \cite[Section~4.2]{contractioncat} for the details. 
Then, Proposition~4.4 in~\cite{contractioncat} implies that, if $\mathcal{M}$ is a ${\bf CGraph}^{\op}_{\leq g}$-module which is a
subquotient of a module ${\bf CGraph}^{\op}_{\leq g}\to \Mod$ (with $R$ here a field) that is finitely generated in degrees $\leq d$, \emph{i.e.}~by graphs with at most $d$ edges, then the dimension of $\mathcal{M}(\tG(\underline{e},\underline{m}))$ is bounded by a polynomial in $\underline{m}$ of degree $\leq d$ -- \emph{c.f.}~\cite[Corollary~4.5]{contractioncat}.  As a consequence, we have the following:

\begin{corollary}\label{cor:hajfgak}
    Let $\bK$ be a field and $\tG$ be a graph of genus $g$. Then, there exists a polynomial $f_{\tG,\underline{e}}(x_1,\dots,x_r)$ of total degree at most $g+k+1$ such that
    \[
    \dim_{\bK} \MH_l^k(\tG(\underline{e},\underline{m});\bK)=  f_{\tG,\underline{e}}(m_1,\dots,m_r) \ ,
\]
provided  $\underline{m}$ is large enough in each entry.
\end{corollary}

\begin{proof}
    The statement directly follows from \cite[Corollary~4.5]{contractioncat} since magnitude cohomology is a subquotient of a module which is finitely generated in degree~$\leq g + k +1$.
\end{proof}

Likewise, a similar result is obtained by considering the functor $\MH_*^k=\bigoplus_l\MH_l^k$, as the finite generation of the magnitude cohomology functor only depends on~$k$ and not on $l$. 

\begin{figure}[h]
    \centering
    \newdimen\R
\R=1.80cm
\begin{tikzpicture}[thick]
\draw[xshift=5.0\R, fill] (270:\R) circle(.05)  node[below] {$v_0$};
\draw[xshift=5.0\R,fill] (225:\R) circle(.05)  node[below left]   {$v_1$};
\draw[xshift=5.0\R,fill] (180:\R) circle(.05)  node[left] {$v_2$};
\draw[xshift=5.0\R,fill] (135:\R) circle(.05)  node[above left] {$v_3$};
\draw[xshift=5.0\R, fill] (90:\R) circle(.05)  node[above] {$v_4$};
\draw[xshift=5.0\R,fill] (45:\R) circle(.05)  node[above right] {$v_5$};
\draw[xshift=5.0\R,fill] (0:\R) circle(.05)  node[right] {$v_6$};
\draw[xshift=5.0\R,fill] (315:\R) circle(.05)  node[below right] {$v_{m}$};

\node[xshift=5.0\R] (v0) at (270:\R) { };
\node[xshift=5.0\R] (v1) at (225:\R) { };
\node[xshift=5.0\R] (v2) at (180:\R) { };
\node[xshift=5.0\R] (v3) at (135:\R) { };
\node[xshift=5.0\R] (v4) at (90:\R) { };
\node[xshift=5.0\R] (v5) at (45:\R) { };
\node[xshift=5.0\R] (v6) at (0:\R) { };
\node[xshift=5.0\R] (vn) at (315:\R) { };

\draw[thick, bunired] (v0)--(v1);
\draw[thick, bunired] (v1)--(v2);
\draw[thick, bunired] (v2)--(v3);
\draw[thick, bunired] (v3)--(v4);
\draw[thick, bunired] (v4)--(v5);
\draw[thick, bunired] (v5)--(v6);
\draw[thick, bunired] (vn)--(v0);

\draw[xshift=5.0\R, fill] (292.5:\R) node[below right] {$e_{m}$};
\draw[xshift=5.0\R,fill] (247.5:\R) node[below left] {$e_0$};
\draw[xshift=5.0\R,fill] (202.5:\R)   node[left] {$e_1$};
\draw[xshift=5.0\R,fill] (157.5:\R)  node[above left] {$e_2$};
\draw[xshift=5.0\R, fill] (112.5:\R)   node[above] {$e_3$};
\draw[xshift=5.0\R,fill] (67.5:\R) node[above right] {$e_4$};
\draw[xshift=5.0\R,fill] (22.5:\R) node[right] {$e_5$};
\draw[xshift=4.95\R,fill] (337.5:\R)  node {$\cdot$} ;
\draw[xshift=4.95\R,fill] (333:\R)  node {$\cdot$} ;
\draw[xshift=4.95\R,fill] (342:\R)  node {$\cdot$} ;
\end{tikzpicture}
\caption{The cycle graph $\tC_m$.} 
    \label{fig:Cm}
\end{figure}
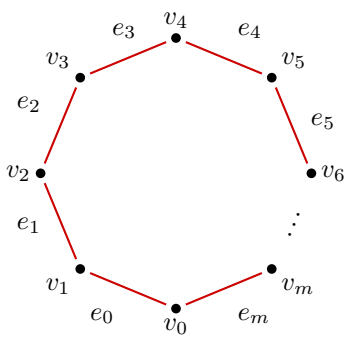

\begin{example}\label{ex:cycle}
Let $\tC_m$ be a cycle graph with $m$ edges -- see Figure~\ref{fig:Cm}. Subdivision of any edge of $\tC_m$ yields again a cycle. The magnitude homology groups of cycles have been computed in~\cite[Theorems~4.6 \&~4.8]{magnitue_discretemorsetheory}, proving a conjecture of~\cite[Appendix~A.1]{richardHHA}. For a fixed $k$, it can be shown that, in their notation,  the dimension satisfies:
\begin{equation}\label{eq:hfahafks}
T^m_{k,l} = \dim_\bZ {\rm MH}_{k,l}(\tC_{m}) = a(k,l) m + b(k,l) \ . 
\end{equation}
Note that $a(k,l), b(k,l) \neq 0$ for finitely many values of $l$, whose number is dependent on $k$. Hence for each $k$ we can set $A(k) = \sum_{l} a(k,l)$ and $B(k) = \sum_{l} b(k,l)$. It follows that $T^m_{k,l} \leq A(k)m + B(k)$.
\end{example}

Example~\ref{ex:cycle} implies that Corollary~\ref{cor:jafbka} and Corollary~\ref{cor:hajfgak} do not provide sharp results. However, they generalise what happens in the aforementioned example to graphs of  fixed or bounded genus.
More precisely, Corollary~\ref{cor:jafbka} is the analogue of the inequality $T^m_{k,l} \leq A(k)m + B(k)$ in Example~\ref{ex:cycle}. Corollary~\ref{cor:hajfgak} is the analogue of the formula  
in Equation~\eqref{eq:hfahafks}, for families of graphs obtained via subdivision of edges.
By  applying~\cite[Corollary~4.7]{contractioncat}, we get a similar statement when considering the operation of ``gluing'' trees to~$\tG$. We refer to \cite{contractioncat} for more applications. 

A second application, of main interest to us, concerns the behaviour of torsion in magnitude (co)homology.  
It was shown in \cite[Theorem 3.14]{SadzanovicSummers} that  any finitely generated Abelian group may appear as a subgroup of the
magnitude homology of a graph, and that there are infinitely many
such graphs. More precisely, Sadzanovic and Summers proved the following:

\begin{theorem}[{\cite[Theorem 3.13]{SadzanovicSummers}}]
Let $p$ be a prime and $n, m \geq 1$ integers. There exist infinitely many distinct
isomorphism classes of graphs whose magnitude homology contains $\bZ_{p^m}$ torsion in bigrading $(3, 2n+
3)$.
\end{theorem}

The proof of \cite[Theorem 3.13]{SadzanovicSummers} is based on Kaneta-Yoshinaga construction~\cite{MR4275098}. Graphs whose integral magnitude homology has ${p^r}$-torsion are obtained from triangulations of (generalised) lens spaces and iterated subdivisions. 
However, there is no structural theorem concerning the complexity of graphs having given torsion in integral magnitude homology. 
The next result suggests that, in order to find more torsion in magnitude (co)homology, one needs to increase the combinatorial complexity of the graphs.

\begin{theorem}\label{thm:torsionmagnitude}
For every pair of integers $k, g\geq 0$, there exists $m = m(g,k) \in \bZ$ which annihilates the torsion subgroup of $\MH^{k}_{*}(\tG;\bZ)$, for each graph~$\tG$ of genus at most $g$. 
\end{theorem}

\begin{proof}
First, note that magnitude homology and magnitude cohomology are related by a universal coefficients short exact sequence by \cite[Remark~2.5]{MagnitudeCohomology}, hence we can restrict to magnitude cohomology (and results for magnitude homology will be derived by application of such short exact sequence). 

Fix a degree $k$, and take $R=\bZ$ the integers. By Corollary~\ref{cor:mgfg}, the ${\bf CGraph}^{\op}_{\leq g}$-module  $\MH^{k}_{l}(-;\bZ)$ is finitely generated. Let $\tG$ be a graph of bounded genus $\leq g$, and $\tau$ a torsion class in $\MH^{k}_{l}(\tG;\bZ)$. For $\phi\colon \tH\to \tG$ a contraction of graphs, we get by functoriality a map in magnitude cohomology that preserves the torsion class. Therefore, we can consider  the submodule $\mathcal{T}\subseteq \MH^{k}_{l}(-;\bZ)$ which sends a graph $\tG$ to the $\bZ$-module $ \MH^{k}_{l}(\tG;\bZ)$. Then, by Corollary~\ref{cor:contractcatfg}, the ${\bf CGraph}^{\op}_{\leq g}$-module $\mathcal{T}$ is also finitely generated. But, by definition, this means that there exist graphs~$\tG_1,\dots,\tG_{m(k)}$ of genus bounded by~$ g$, and a surjection $\bigoplus_{i=1}^{m(k)}\mathcal{P}_{\tG_i}\to \mathcal{T}$ from the associated principal projectives. Now, choose $N$ to be the least common multiple of the annihilators of $\mathcal{T}(\tG_1),\dots, \mathcal{T}(\tG_{m(k)})$; then, for any graph~$\tG$ in ${\bf CGraph}^{\op}_{\leq g}$, the torsion part of $\MH^{k}_{l}(\tG;\bZ)$ has exponent at most $N$. This concludes the proof.
\end{proof}

From Remark~\ref{rem:pathhom} and \cite{asao} follows that (reduced) path homology, as introduced in \cite{Grigoryan_first}, appears as the diagonal in the second page of a spectral sequence whose $0$-th page features magnitude chain groups. Shifting to later pages in the spectral sequence is obtained by subsequent subquotients. Therefore, after restricting to undirected graphs, Theorem~\ref{thm:Vkfg} yields, with the same proof of Theorem~\ref{thm:torsionmagnitude}, the following:

\begin{corollary}\label{cor:torionPH}
For each $g,k$ positive integers, there exists a $d = d({g,k}) \in \bZ$ such that, for each graph $\tG$ of genus $g$, the torsion part of the path cohomology $\mathrm{PH}^k(\tG,\bZ)$  has exponent at most $d$.
\end{corollary}

\bibliographystyle{alpha}
\bibliography{bibliography}

\end{document}